\documentclass{amsproc}
\usepackage[utf8]{inputenc}
\usepackage[T1]{fontenc}
\usepackage{lmodern}
\usepackage{graphicx,amsmath,amsfonts,latexsym,amsthm,amssymb,mathrsfs,mathtools,stmaryrd}
\usepackage[abs]{overpic}
\usepackage[dvipsnames]{xcolor}
\usepackage[linktocpage=true]{hyperref}
\usepackage{caption}
\usepackage{dsfont}
\usepackage{tikz-cd}
\usepackage{comment}
\usepackage{enumitem}
\usepackage{microtype}
\numberwithin{equation}{section}

\usepackage{graphpap}
\usepackage{svg}
\newtheorem{thm}{Theorem}[section]
\newtheorem{cor}[thm]{Corollary}
\newtheorem{lem}[thm]{Lemma}
\newtheorem{question}[thm]{Question}
\newtheorem{prop}[thm]{Proposition}
\newtheorem{conj}[thm]{Conjecture}
\newtheorem*{conj*}{Conjecture}

\newtheorem{defn}[thm]{Definition}
\newtheorem{ex}[thm]{Example}

\newtheorem*{thm*}{Theorem}
\newtheorem*{cor*}{Corollary}

\newcommand{\Z}{\mathbb{Z}}
\newcommand{\R}{\mathbb{R}}

\newcommand{\Q}{\mathbb{Q}}

\title{Taut foliations through a contact lens}

\author{Thomas Massoni}
\address{Department of Mathematics, Stanford University, 450 Jane Stanford Way, Building 380, Stanford, USA}
\email{tmassoni@stanford.edu}


\begin{document}

\begin{abstract}
This survey explores the rich interplay between (taut) foliations and (tight) contact structures in dimension three, highlighting recent work of the author. We outline a new method for constructing taut foliations from suitable pairs of contact structures, and discuss some applications and future research directions. In particular, we give a brief account of work in progress with Jonathan Zung on \emph{ziggurats} for taut foliations transverse to pseudo-Anosov flows, and we propose a (mostly speculative) contact perspective on the $L$-space conjecture.
\end{abstract}

\maketitle

\tableofcontents

\section{Introduction}

Foliations and contact structures are fundamental geometric structures on $3$-manifolds, which can both be defined in terms of suitable plane fields. While foliations correspond to integrable plane fields, tangent to $2$-dimensional surfaces or leaves partitioning the manifold, contact structures are maximally nonintegrable, very ``twisty'' plane fields which are nowhere tangent to surfaces. 

In dimension three, both foliation theory and contact topology exhibit a rich interaction with low-dimensional topology, dynamics, and geometry. Of particular interest are the ``rigid'' objects, namely \emph{taut} foliations and \emph{tight} contact structures, which carry strong topological information and play a central role in various applications to dynamical systems, knot theory, gauge theory, and so on.

While foliations and contact structures are fundamentally different objects, there exist many striking parallels between the two theories: existence results, $h$-principles, the ``flexible versus rigid'' dichotomy, the sutured theory, to cite a few. We refer the reader to the beautiful survey~\cite{CH20} by Colin and Honda for an overview of these interactions. This naturally prompts the question: can one construct contact structures from foliations, and vice versa, and establish a direct correspondence between the two worlds?

The first general result connecting foliations and contact structures was obtained by Eliashberg and Thurston in~\cite{ET}: any suitable foliation in dimension $3$ can be approximated by contact structures, in the sense that its tangent plane field can be perturbed into contact structures by some arbitrarily small $C^0$ modification. Moreover, contact approximations to taut foliations are automatically tight, yielding a very general construction of tight contact structures in dimension $3$. This result led to many other breakthroughs, such as the proof by Ozsváth and Szabó that closed $3$-manifolds carrying coorientable taut foliations are not Heegaard Floer $L$-spaces, namely, their Heegaard Floer homology is ``complicated''.

The connection in the other direction, from contact structures to foliations, was initiated by Mitsumatsu~\cite{M95} and Eliashberg--Thurston~\cite{ET} in the special case of contact structures coming from Anosov flows, and by Colin and Firmo~\cite{CF11} in a more general setting. In~\cite{Mas24}, the author established a complete converse to the Eliashberg--Thurston theorem, extending the work of Colin--Firmo, and constructing foliations from contact structures in full generality. The key input is a \emph{pair} of contact structures with opposite signs, in a suitable geometric position with respect to each other---those can be thought of as the positive and negative ``contact shadows'' of the foliation.

This contact viewpoint on foliation theory provides some degree of flexibility, as contact structures are easier to deform, and allows us to perform some surgery operations that are not directly available for foliations. Those techniques are particularly powerful for studying foliations transverse to a fixed flow on a manifold with boundary, as it is more natural to modify a contact structure near the boundary components than a foliation.

The correspondence between foliations and contact pairs also offers a new perspective on the \emph{$L$-space conjecture}. Since contact structures are more closely related to gauge theory than foliations, a reasonable strategy would be to first attempt to construct suitable contact structures under relevant gauge-theoretic hypotheses, and then use them to produce taut foliations. This program is for now mostly conjectural, and requires a robust geometric theory of \emph{pairs} of contact structures which is yet to be fully developed.

    \subsection*{Acknowledgments}

I would like to thank Jonathan Bowden and Jonathan Zung for numerous conversations about foliations and contact structures, and for teaching me most of what I know on these topics. I am very grateful to the 2025 Georgia International Topology Conference for giving me the opportunity to present my research and to contribute to its proceedings. I would also like to thank the anonymous referee for their helpful comments and suggestions. I am supported by a Stanford Science Fellowship.

\section{Foliations and contact structures}

In this survey, we are mainly interested in two types of structures on $3$-manifolds: foliations and contact structures. Both of them can be defined in terms of plane fields, i.e., rank $2$ distributions on the manifold.

    \subsection{Plane fields}

We begin with a general discussion of $2$-dimensional plane fields on $3$-manifolds.

\subsubsection{Definition}

A \textbf{plane field} on $M$ is the data of a $2$-dimensional subspace $\eta_p \subset T_p M$ for each $p \in M$. It can be described \emph{locally} as the span of two linearly independent vector fields, or dually, as the kernel of a $1$-form $\alpha$: $\eta_p = \mathrm{ker}(\alpha_p)$. It is of class $C^k$, $k \in \Z_{\geq 0}$, if such an $\alpha$ can be chosen to be $C^k$ (at least locally). A continuous plane field $\eta$ is \textbf{(co)orientable} if there exist consistent choices of (co)orientations for the $\eta_p$, $p \in M$. Note that $\eta$ is coorientable if and only if there exists a \emph{globally} defined (continuous or smooth) vector field $Z$ transverse to $\eta$, or dually, if there exists a globally defined $1$-form $\alpha$ as before. If $M$ itself is orientable, then $\eta$ is orientable if and only if it is coorientable.

\subsubsection{Tangent and transverse surfaces}

Let $\eta$ be a continuous plane field on a $3$-manifold $M$. An immersion $f:\Sigma \rightarrow M$ defined on a (smooth) surface $\Sigma$ is \textbf{tangent to $\eta$} if for every $z \in \Sigma$, we have $\mathrm{range}(d_zf) = \eta_{f(z)}$. It is \textbf{transverse to $\eta$} if for every $z \in \Sigma$, $\mathrm{range}(d_zf) \pitchfork \eta_{f(z)}$. In that case, the formula $\ell_z \coloneqq (d_zf)^{-1}\big(\eta_{f(z)}\big)$ for $z \in \Sigma$ defines a continuous line field on $\Sigma$ called the \textbf{characteristic line field} of $f$ for $\eta$.

\subsubsection{Tautness}

The notion of \emph{tautness} is usually defined for foliations, but we slightly extend this definition to continuous plane fields. We say that a continuous, cooriented plane field $\eta$ on $M$ is \textbf{taut} if for every $p \in M$, there exists a (smooth) loop $\gamma : S^1 \rightarrow M$ based at $p$ which is transverse to $\eta$. We say that a codimension-$0$ submanifold $N \subset M$ with $C^1$ boundary is a \textbf{dead-end component} of $\eta$ if the boundary components of $N$ are all tangent to $\eta$, and the coorientation of $\eta$ is pointing \emph{inwards} along all of those boundary components. 

As for foliations, there are many equivalent characterizations of tautness for continuous plane fields:

\begin{prop}[See~\cite{Mas24}, Proposition B.3] \label{prop:tautequiv}
Assume that $M$ is oriented. Let $\eta$ be a (co)oriented continuous plane field on $M$. Then the following properties are equivalent:
    \begin{enumerate}
        \item $\eta$ is taut.
        \item For all $p, q \in M$, there exists a smooth path from $p$ to $q$ in $M$ which is transverse to $\eta$.
        \item $\eta$ has no dead-end components.
        \item There exists a (smooth) \emph{closed} $2$-form $\omega$ on $M$ such that $\omega_{\vert \eta} > 0$, i.e., $\omega$ is nondegenerate along $\eta$ and compatible with the orientation on $\eta$.
        \item There exists a (smooth) volume preserving vector field on $M$ which is transverse to $\eta$.
    \end{enumerate}    
\end{prop}

Here, we say that a vector field $Z$ is \textbf{volume preserving} if there exists a (smooth) volume form $\mathrm{dvol}$ on $M$ preserved by the flow of $Z$, i.e., satisfying $\mathcal{L}_Z \mathrm{dvol} = 0$. Equivalently, this last equality means that the $2$-form $\omega = \iota_Z \mathrm{dvol}$ is \emph{closed}.

It immediately follows from item 4 or item 5 that tautness is an \emph{open condition} in the $C^0$ topology for continuous plane fields. Moreover, by item 3, a continuous oriented plane field without closed tangent surfaces is automatically taut.

    \subsection{Foliations}

In this article, we will focus on foliations admitting a well-defined, continuous tangent plane field. We give a slightly nonstandard account of the theory, and refer to Calegari's book~\cite{C07} and Candel--Conlon's encyclopedia~\cite{CC1,CC2} for much more on foliations. Let us first recall some basic definitions.

\subsubsection{Three shades of integrability}

Let $\eta$ be a continuous plane field on a closed $3$-manifold $M$. 

We say that $\eta$ is \textbf{locally integrable} if for every $p \in M$, there exists an embedded surface $\Sigma \hookrightarrow M$ passing through $p$ which is tangent to $\eta$. In this definition, we can assume that $\Sigma$ is a disk. 

Furthermore, we say that $\eta$ is \textbf{uniquely integrable near $p \in M$} if the images of any two tangent surfaces passing through $p$ coincide near $p$. Then, we say that $\eta$ is \textbf{uniquely integrable} if it is uniquely integrable near every point $p \in M$. 

We say that $\eta$ is \textbf{integrable}, or \textbf{tangent to a foliation} if there exists a partition of $M$ into injectively (but not necessarily properly!) immersed surfaces 
\begin{align} \label{eq:partition}
    M = \bigsqcup_{i \in I} L_i 
\end{align}
such that every $L_i$, $i \in I$, is tangent to $\eta$ and is \emph{maximal} in the following sense: if there is another injectively immersed surface $L_i \subset L \subset M$ tangent to $\eta$ containing $L_i$, then $L_i = L$. Alternatively, this means that $L_i$ is \emph{complete} for the pullback of some/any Riemannian metric on $M$. It follows that this decomposition is \emph{locally trivial}, in the sense that near every $p \in M$, there exist smooth local coordinates $(x, y, z) \in (-1,1)^3$ so that $\eta$ is transverse to $\partial_z$ in these coordinates, and the intersection of each $L_i$ with a smaller neighborhood of $p$ is a disjoint union of graphs of $C^1$ maps $(-\delta,\delta)^2_{x,y} \rightarrow (-1,1)_z$.

One has:

\begin{lem}
    If $\eta$ is a continuous, uniquely integrable plane field on a closed $3$-manifold $M$, then it is tangent to a foliation.
\end{lem}

\begin{proof}
    We only sketch the main ideas. We define an equivalence relation on $M$ as follows: two points $p, q \in M$ are equivalent with respect to $\eta$, denoted by $p \sim_\eta q$, if there exists a $C^1$ path from $p$ to $q$ tangent to $\eta$. Then, one shows that this is a well-defined equivalence relation on $M$, and that each equivalence class is an injectively immersed surface in $M$, using the unique integrability condition. The partition~\eqref{eq:partition} is simply the partition of $M$ into equivalence classes for $\sim_\eta$. The maximality of each $L_i$ also follows from the unique integrability condition and the compactness of $M$.
\end{proof}

By the Cauchy--Lipschitz--Picard--Lindel\"{o}f theorem, any locally integrable and Lipschitz regular (e.g., $C^1$ regular) plane field is uniquely integrable, hence tangent to a foliation. Moreover, by the Clebsch--Deahna--Frobenius theorem, a $C^1$ plane field $\eta$ is uniquely integrable if and only if it satisfies the involutivity condition with respect to the Lie bracket: for any two (local) $C^1$ vector fields $X, Y \in \eta$ tangent to $\eta$, their Lie bracket $[X, Y]$ is also tangent to $\eta$. Dually, if $\eta = \ker \alpha$ for a $C^1$ (local) $1$-form $\alpha$, the involutivity condition is equivalent to 
\begin{align}
    \alpha \wedge d\alpha \equiv 0,
\end{align}
as a consequence of Cartan magic formula. However, a \emph{continuous} and locally integrable plane field might fail to be tangent to a foliation. Here is a $2$-dimensional example.

\begin{ex}
Let $f : \R \rightarrow \R$ be the continuous function defined by 
$$f(x) \coloneqq \mathrm{sign}(x) \vert x \vert^{1/2},$$
where $\mathrm{sign}(x) = 1$ if $x> 0$, $\mathrm{sign}(x)= -1$ if $x< 0$, and $\mathrm{sign}(0)=0$. Let $X$ be the continuous vector field on $\R^2_{x,y}$ defined by
$$X \coloneqq \partial_x + f(y)\,   \partial_y.$$
The flow lines of $X$ correspond to the solutions to the ODE 
$$y' = f(y),$$
which are of the form
$$y(x) = \begin{cases}
            0 & \mathrm{if \ } x \leq c, \\
            \sigma \frac{1}{4} (x-c)^2 & \mathrm{if \ } x \geq c,
         \end{cases}$$
for $\sigma \in \{-1, 0, 1\}$ and some $c \in \R$. Therefore, the line field generated by $X$ is locally integrable, as every point in $\R^2$ is contained in a flow line of $X$. However, any two flow lines of $X$ intersect along a set of the form $\{ x \leq c, \ y=0\}$ for some $c \in \R \cup \{+ \infty\}$. Therefore, this line field is not tangent to a foliation, and in particular, it is not uniquely integrable. Notice that $X$ is H\"older continuous but not Lipschitz.
\end{ex}

In general, it is hard to determine if a continuous plane field is locally integrable, as it amounts to solving nonlinear PDEs with continuous coefficients. However, this can be reduced to a $1$-dimensional problem (i.e., solving an ODE) in the presence of sufficiently regular tangent vector fields preserving $\eta$.

\begin{lem} \label{lem:tangentint}
    Let $\eta$ be a continuous plane field on a $3$-manifold $M$, and assume that there exists a smooth (or $C^1$) nowhere vanishing vector field $X$ tangent to $\eta$ and whose flow preserves $\eta$. Then $\eta$ is locally integrable.
\end{lem}

\begin{proof}
    Let $p \in M$. We construct an embedded disk in $M$ tangent to $\eta$ and passing through $p$ as follows. First, we choose a flow box for $X$ around $p$, i.e., smooth (or $C^1$) coordinates $(x, y, z) \in (-\delta, \delta)^3$ centered at $p$ in which $X$ becomes $\partial_x$. We can further assume that $\partial_z$ is transverse to $\eta$ in these coordinates. Since $\eta$ is invariant under $X$, it is of the form
    $$\eta = \mathrm{span}\{\partial_x, \partial_y + f(y,z) \partial_z \},$$
    for some continuous function $f$ which is independent of $x$. Note that the vector field $\partial_y + f(y,z) \partial_z$ spans the characteristic line field of each surface $\{x = cst\}$ for $\eta$. By the Cauchy--Peano theorem, the ODE
    $$z'(y) = f(y,z), \qquad z(0)=0,$$
    admits a (not necessarily unique!) solution defined in a neighborhood of $0$. In other words, there exists an $\epsilon > 0$ and a $C^1$ curve $\gamma : (-\epsilon, \epsilon) \rightarrow \R^2_{y,z}$ with $\gamma(0) = (0,0)$, and which is tangent to $\partial_y + f(y,z) \partial_z$. Then, $(-\epsilon, \epsilon)_x \times \gamma$ is a $C^1$ embedded disk passing through $(0,0,0)$ and tangent to $\eta$.
\end{proof}

\subsubsection{$C^0$-foliations}

We now define a \textbf{$C^0$-foliation} as a collection of injectively immersed ($C^1$) surfaces partitioning $M$, which are tangent to a (globally defined) continuous plane field, and which are maximal as before. We call these surfaces the \textbf{leaves} of the foliation. While a $C^0$-foliation has a well-defined continuous tangent plane field by definition, a given continuous plane field might be tangent to many \emph{different} foliations, due to the lack of unique integrability. See~\cite{BF03} for some striking examples of this phenomenon in dimension $2$.

One can also define foliations through the notion of \emph{foliated atlases}, without requiring the existence of a tangent plane field. The leaves of such topological foliations are not necessarily $C^1$. However, a theorem of Calegari implies that any topological foliation by surfaces on a $3$-manifold can be isotoped to a $C^0$-foliation whose leaves are moreover smoothly immersed (but the tangent plane field only varies continuously in the direction transverse to the leaves), see~\cite{C01}.

\subsubsection{Tautness} Let $\mathcal{F}$ be a $C^0$-foliation, with tangent plane field denoted by $T \mathcal{F}$. We say that $\mathcal{F}$ is \textbf{(everywhere) taut} if $T \mathcal{F}$ is taut as a $C^0$-plane field. We warn the reader that there are weaker definitions of tautness in the literature which are not all equivalent for $C^0$-foliations, but become equivalent after suitable isotopies, see~\cite{CKR19}.

An important consequence of tautness due to Novikov is that transverse loops to a taut foliation are noncontractible. In particular, the fundamental group of the ambient manifold is nontrivial. Moreover, the universal cover of the manifold is irreducible (except if the foliation is homeomorphic to the fibration by spheres on $S^1 \times S^2$). See~\cite{C07} for a more thorough discussion of taut foliations in dimension $3$.

    \subsection{Contact structures}

We now provide a very quick overview of some basic notions of contact geometry in dimension $3$. We refer the reader to~\cite{G09} for a more thorough introduction.

\subsubsection{Nonintegrability}

A smooth (or $C^1$) plane field $\xi$ on a $3$-manifold $M$ is a \textbf{contact structure} if it is ``maximally nonintegrable'', in the sense that for any two (germs of) linearly independent vector fields $X$ and $Y$ spanning $\xi$ near a point $p \in M$, we have $[X,Y](p) \notin \xi$. Dually, this means that $\xi$ is (locally) defined as $\xi = \ker \alpha$ for a $1$-form $\alpha$ satisfying that $\alpha \wedge d\alpha$ is nowhere vanishing near $p$. As a consequence, for every $p \in M$, there are no (germs of) surfaces passing through $p$ and tangent to $\xi$. In other words, $\xi$ is ``maximally twisting''. This can be rephrased in terms of tangent vector fields as follows.

\begin{lem}
    Let $\xi$ be a smooth (or $C^1$) plane field tangent to a smooth (or $C^1$) nonvanishing vector field $X$. Let $g$ be some auxiliary Riemannian metric, and let $\theta_t(p) \in [0,\pi)$ denote the unoriented angle between $\xi(p)$ and $\xi_t(p) \coloneqq (\varphi_X^t)_* \xi(p)$, which is well-defined for $\vert t \vert$ small enough. Then $\xi$ is contact if and only if for all $p \in M$ and all small $t \geq 0$,
    $$\partial_t \theta_t(p) > 0.$$
\end{lem}

See~\cite[Proposition 1.1.6]{ET} for a proof. This means that $\xi$ twists along any vector field tangent to it; compare with Lemma~\ref{lem:tangentint}.

When the ambient manifold $M$ is oriented and the contact structure $\xi$ itself is (co)oriented, one can choose a global $1$-form $\alpha$ satisfying $\xi = \ker \alpha$, and which is compatible with the coorientation of $\xi$. Such an $\alpha$ is called a \textbf{contact form} for $\xi$. Then $\alpha \wedge d\alpha$ is a volume form on $M$ and has a well-defined \emph{sign}. Any other contact form $\alpha'$ for $\xi$ is of the form $f \alpha$, for a positive function $f$, and $\alpha' \wedge d\alpha' = f^2 \alpha \wedge d\alpha$, so this sign is independent of the choice of a specific contact form for $\xi$. We say that $\xi$ is a \textbf{positive (resp.~negative)} contact structure if this sign is positive (resp.~negative). 

Let us make this discussion more concrete. On $\R^3_{x,y,z}$, the \textbf{standard (positive and negative) contact structures} are defined by
\begin{align}
    \xi^\mathrm{std}_+ &\coloneqq \{dz + x dy =0 \},\\
    \xi^\mathrm{std}_- &\coloneqq \{dz - x dy =0 \}.
\end{align}
One can alternatively define them in cylindrical coordinates as $\{ dz \pm r^2 d\theta =0\}$. It is an instructive exercise to find a diffeomorphism of $\R^3$ sending $\xi^\mathrm{std}_+$ to $\{ dz + r^2 d\theta =0\}$.

\subsubsection{Basic properties}

A fundamental property of contact structures is that they are \emph{locally trivial}. Indeed, \textbf{Darboux theorem} asserts that if $\xi$ is a contact structure on a $3$-manifold $M$, then for every $p \in M$, there exist coordinates $(x,y,z)$ centered at $p$ in which $\xi$ coincides with $\{ dz + xdy = 0\}$. If $M$ and $\xi$ are oriented, and if $\xi$ is positive, we can further assume that these coordinates preserve the orientation on $M$.

Another remarkable feature of contact structures is that they are \emph{stable} objects, in the sense of \textbf{Gray stability}: if $(\xi_t)_{t \in [0,1]}$ is a path of contact structures on $M$ and $M$ is closed, then there exists an isotopy $(\phi_t)_{t \in [0,1]}$ of $M$ such that $\phi_t^*\xi_t = \xi_0$ for every $t \in [0,1]$. In particular, if $\xi$ is a given contact structure on $M$ and if $\xi'$ is a $C^1$ plane field which is sufficiently $C^1$-close to $\xi$, then $\xi'$ is also a contact structure and it is isotopic to $\xi$. Notice that there is no equivalent statement for integrable plane fields!

Since the previous results hold for $C^1$ contact structures, every $C^1$-contact structure can be approximated by a smooth one which is moreover $C^1$-isotopic to it. Therefore, we will always assume that contact structures are smooth.

\subsubsection{Tightness}

In $\R^3$ with cylindrical coordinates, we define the standard \textbf{overtwisted} contact structure $\xi_\mathrm{OT}$ as
$$\xi_\mathrm{OT} \coloneqq \{\cos(r)dz + r \sin(r) d\theta = 0\}.$$
Note that $\xi_\mathrm{OT}$ is tangent to the disk $D_\mathrm{OT} \coloneqq \{r \leq \pi, \ z=0\}$ along $\partial D_\mathrm{OT}$ and at the origin, and is transverse to it elsewhere, where the characteristic foliation is made of rays from the origin to the boundary. An \textbf{overtwisted disk} in a contact $3$-manifold $(M, \xi)$ is a contact embedding of a neighborhood of $(D_\mathrm{OT}, \xi_\mathrm{OT})$ in $M$. If such an overtwisted disk exists, we call $\xi$ \textbf{overtwisted}. Otherwise, we call it \textbf{tight}.

It was shown by Bennequin that the standard positive and negative contact structures on $\R^3$ are tight. A fundamental result of Eliashberg~\cite{E89} implies that overtwisted contact structures are completely classified in terms of their homotopy class as plane fields. On the other hand, tight contact structures are more geometric and are much harder to construct and classify.

\section{From foliations to contact structures}

Foliations and contact structures are, by definition, quite different objects; foliations arise from integrable plane fields, whereas contact structures are maximally nonintegrable plane fields. Nevertheless, these two types of structures are more closely related than it may seem.

    \subsection{The Eliashberg--Thurston theorem}

\subsubsection{Statement}

The first general result relating foliations and contact structures was obtained by Eliashberg and Thurston in~\cite{ET}. They proved that most foliations on a closed oriented $3$-manifold can be $C^0$-approximated by contact structures. More precisely:

\begin{thm}[Eliashberg--Thurston~\cite{ET}]
    Let $\mathcal{F}$ be an orientable foliation on a closed, oriented $3$-manifold $M$, different from the product foliation by spheres on $S^1 \times S^2$. Assume that $T\mathcal{F}$ is $C^2$. Then $T\mathcal{F}$ can be $C^0$-approximated by positive and negative contact structures. Moreover, if $\mathcal{F}$ is taut, then its contact approximations are tight.
\end{thm}

Note that by the Reeb stability theorem, the only foliation on a closed orientable $3$-manifold with a sphere leaf is the product foliation by spheres on $S^1 \times S^2$. Foliations without sphere leaves are called \textbf{aspherical}.

\subsubsection{Glimpse into the proof}

Let us give a very brief overview of the proof of the Eliashberg--Thurston theorem. The idea is to successively approximate $T \mathcal{F}$ by plane fields which interpolate between foliations and contact structures, called \textbf{confoliations}: a $C^1$ plane field $\xi$ is a (cooriented, positive) confoliation if it is of the form $\xi = \ker \alpha$ for a $1$-form $\alpha$ satisfying $\alpha \wedge d\alpha \geq 0$. Then, $\xi$ is contact on the open set where $\alpha \wedge d\alpha > 0$. If this set is sufficiently large, then one can ``propagate'' the contactness along curves tangent to $\xi$ and make it contact everywhere, by a $C^0$-small perturbation. The key is then to perturb $T \mathcal{F}$ to create some contactness on a neighborhood of its \emph{minimal sets}. A minimal set of $\mathcal{F}$ is a (nonempty!) closed subset $\Lambda \subset M$ which is saturated by leaves of $\mathcal{F}$, and which is minimal for the inclusion among those sets. Note that every point in $M$ can be joined to a neighborhood of some minimal set by a curve tangent to $\mathcal{F}$. Minimal sets are of three types: closed leaves, $M$ itself if it contains a dense leaf of $\mathcal{F}$, and \emph{exceptional} minimal sets, which have a special structure. Assuming that the foliation is not a foliation without holonomy,\footnote{The case of foliations without holonomy is treated differently; these foliations are essentially classified and can be perturbed ``by hand''.} and using results of Sacksteder and Ghys (and ad hoc modifications near closed leaves), one can approximate $\mathcal{F}$ by a foliation whose minimal sets all admit a curve with \emph{attracting holonomy}. Such a curve $\gamma : S^1 \rightarrow M$ is a $C^1$ curve contained in a leaf $L$ of $\mathcal{F}$, so that the holonomy along $\gamma$ is contracting: the leaves parallel to $L$ get closer to $L$ while going along $\gamma$. Then one can approximate the latter foliation by confoliations which are contact near curves with attracting holonomy. Here is a model computation:

\begin{ex}
    Let $\mathcal{F}$ be the foliation on $S^1_x \times \R^2_{y,z}$ defined by the $1$-form
    $$\alpha_0 \coloneqq dz + f(z) dx,$$
    where $f : \R \rightarrow \R$ is a $C^1$ map satisfying $f(0)=0$ and $f' > 0$. Note that we have $T \mathcal{F} = \mathrm{span}\big( \partial_y, \partial_x - f(z) \partial_z\big)$, so the curve $\gamma \coloneqq S^1 \times \{(0,0)\}$ is a curve with attracting holonomy along the leaf $L_0 = \{z=0\}$ of $\mathcal{F}$. We can write an explicit perturbation of $T \mathcal{F}$ into a positive confoliation, following~\cite[Proposition 2.6.1]{ET}. First, let $\tau : \R_{\geq 0} \rightarrow  \R_{\geq 0}$ be a smooth cutoff function supported near $0$ which satisfies $\tau' \leq 0$, and $\tau(t) = 1$ for $t$ small enough. We then define $\rho(y,z) \coloneqq \tau(y^2 + z^2)$, and for $\epsilon > 0$, we set
    $$\alpha_+ \coloneqq \alpha_0 + \epsilon \rho \, dy.$$
    We compute:
    \begin{align*}
        \alpha_+ \wedge d\alpha_+ &= \epsilon \underset{u(y,z) \geq 0}{\underbrace{\big( \rho f' - 2 zf \tau'(y^2 + z^2) \big)}} \, dx \wedge dy \wedge dz,
    \end{align*}
    so that $\alpha_+$ is a positive confoliation. Moreover, for $\vert y\vert$ and $\vert z\vert$ small enough we have $u(y,z) = f'(z) > 0$, so $\alpha_+$ is contact in a neighborhood of $\gamma$. Outside a larger neighborhood of $\gamma$, $\alpha_+ = \alpha_0$, so the confoliation coincides with $T \mathcal{F}$ away from $\gamma$.
\end{ex}

After these operations, we obtain a confoliation $\xi$ approximating $T \mathcal{F}$ which is moreover \emph{transitive}: any point in $M$ can be joined to the contact region by a curve tangent to the foliation. The contactness can then be propagated along neighborhoods of curves tangent to $\xi$; see~\cite[Proposition 2.8.1]{ET} for details.

\smallskip

If the foliation $\mathcal{F}$ is taut, then there exists a closed $2$-form $\omega$ that evaluates positively along its leaves. Eliashberg and Thurston construct a symplectic structure on $[-1,1] \times M$, which is a \emph{weak symplectic filling} of $(M, \xi_-) \sqcup (M, \xi_+)$, where $\xi_\pm$ are negative and positive contact approximations of $\mathcal{F}$. Concretely, one defines a $2$-form $\Omega$ on $[-1,1]_t \times M$ by
$\Omega \coloneqq \omega + d(t \alpha)$,
where $\alpha$ is a $1$-form defining $\mathcal{F}$. By definition, $\Omega$ is closed, and one easily checks that it is nondegenerate, i.e., that $\Omega \wedge \Omega > 0$. Moreover, if $\xi_\pm$ are positive and negative contact approximations of $\mathcal{F}$ which are sufficiently close to $T \mathcal{F}$, then $\Omega$ restricts positively along $\xi_\pm$ on $\{\pm 1 \} \times M$. In this context, a theorem of Eliashberg and Gromov implies that the contact approximations $\xi_-$ and $\xi_+$ are tight.

    \subsection{Generalizations and variations}

The Eliashberg--Thurston theorem provides an efficient way to construct (tight) contact structures by approximating (taut) foliations. It has been extended and generalized in various directions; we now summarize some of them.

\subsubsection{$C^0$-foliations}

The Eliashberg--Thurston theorem was generalized to $C^0$-foliations independently by Kazez--Roberts~\cite{KR17} and Bowden~\cite{B16a}. Together with the smoothing results in~\cite{C01} and~\cite{KR19}, this implies that any closed oriented $3$-manifold admitting a coorientable taut \emph{topological} foliation admits positive and negative tight contact structures. By a deep result of Gabai, any irreducible closed $3$-manifold with positive first Betti number admits such a taut foliation, hence tight contact structures.

\subsubsection{Reebless foliations}

A foliation has a \emph{Reeb component} if it admits a torus leaf bounding a solid torus such that all the leaves inside that torus are planes. In particular, this solid torus (or its complement) is a dead-end component. Colin~\cite{C02} showed that \emph{any} $C^2$ aspherical \emph{Reebless} foliation $\mathcal{F}$ admits tight contact approximations. This has been strengthened by Bowden who showed that \emph{any} $C^0$ small contact approximation of $\mathcal{F}$ is tight~\cite{B16b}. Moreover, Colin's result holds for Reebless $C^0$-foliations~\cite{B16a} as well, with the caveat that while they admit tight contact approximations, it is not known if \emph{every} sufficiently small contact approximation is tight.

\subsubsection{Foliations without invariant transverse measures}

More recently, Zung showed in~\cite{Z21} that a $C^2$ foliation \emph{without holonomy-invariant transverse measures} $\mathcal{F}$ can be approximated by contact structures whose Reeb vector fields are transverse to $\mathcal{F}$. Since $\mathcal{F}$ is taut, the closed orbits of these Reeb vector fields are noncontractible, so these contact approximations are \emph{hypertight}. 

\subsubsection{Uniqueness of contact approximations}

In~\cite{V16}, Vogel proved that for most $C^2$ foliations without torus leaves, all the (positive) contact structures in a sufficiently small neighborhood of $T\mathcal{F}$ are contact isotopic. Therefore, the contact approximations of Eliashberg--Thurston are unique for these foliations. However, some foliations on $\mathbb{T}^3$ have infinitely many pairwise non-isotopic contact approximations, distinguished by their \emph{Giroux torsion}.

\subsubsection{Topological invariance}

In~\cite{BM}, Bowden and the author proved a strengthening of Vogel's uniqueness result, by showing that the contact approximations to a suitable foliation behave well under topological conjugations of foliations. More precisely, if $\mathcal{F}_0$ and $\mathcal{F}_1$ are two foliations for which Vogel's uniqueness applies, and if $h : M \rightarrow M$ is a \emph{homeomorphism} sending the leaves of $\mathcal{F}_0$ to the leaves of $\mathcal{F}_1$, then the positive (resp.~negative) contact approximations of $\mathcal{F}_0$ and $\mathcal{F}_1$ are contactomorphic.

\subsubsection{Approximation versus deformation}

A \emph{contact deformation} of a $C^0$-foliation $\mathcal{F}$ is a $C^0$-continuous $1$-parameter family of plane fields $(\xi_t)$, $t \in [0,1]$, where $\xi_0 = T \mathcal{F}$, and for $t > 0$, $\xi_t$ is a contact structure. While $C^2$ foliations `with enough holonomy' can be deformed into contact structures (see~\cite[Proposition 2.9.2]{ET}), linear foliations without closed leaves on $\mathbb{T}^3$ do \emph{not} admit contact deformations (see~\cite[Corollary 9.11]{V16}). In contrast, Etnyre showed that any (cooriented) contact structure on a $3$-manifold is a smooth deformation of a foliation, see~\cite{E07}. However, the foliations constructed by Etnyre \emph{always} have Reeb components.

    \subsection{Contact pairs}

Remarkably, the Eliashberg--Thurston theorem and its variations naturally produce \emph{two} contact approximations: a positive and a negative one. We shall think of those, and the way they interact, as capturing some geometric information about the original foliation. We will need both of them to be able to `reconstruct' the foliation, or rather produce a new one that is closely related to it. This motivates the following definition:

\begin{defn} \label{def:contpair}
A \textbf{contact pair} on an oriented $3$-manifold $M$ is a pair $(\xi_-, \xi_+)$ where $\xi_-$ is a negative contact structure and $\xi_+$ is a positive contact structure. Both are cooriented. A contact pair is \textbf{positive} if there exists a vector field $Z$ positively transverse to $\xi_-$ and $\xi_+$.
\end{defn}

Alternatively, a contact pair $(\xi_-, \xi_+)$ is positive if and only if at every point $p \in M$ with $\xi_-(p) = \xi_+(p)$, these two planes have the same orientation (and opposite coorientations). Positive contact pairs were introduced by Colin and Firmo in~\cite{CF11}. See also~\cite{CH20} for a summary (in English!) of Colin--Firmo's paper. One can also adapt the various definitions of tautness to positive contact pairs. For instance, a positive contact pair $(\xi_-, \xi_+)$ is \textbf{taut}\footnote{Those pairs were called \emph{strongly tight} in~\cite{CH20} and \emph{`fortement tendues'} in~\cite{CF11}. We found this terminology quite cumbersome and the name `taut' more appropriate.} if for every point $p \in M$, there exists a closed loop positively transverse to both $\xi_\pm$ and passing through $p$. In~\cite{Mas24}, we discussed some equivalent formulations:

\begin{prop} \label{prop:strongtight}
Let $(\xi_-, \xi_+)$ be a positive contact pair on $M$. The following are equivalent:
\begin{enumerate}
    \item $(\xi_-, \xi_+)$ is taut.
    \item For all $p, q \in M$, there exists a smooth path from $p$ to $q$ that is positively transverse to both $\xi_\pm$.
    \item There exists a smooth closed $2$-form $\omega$ such that $\omega_{\vert \xi_\pm} > 0$.
    \item There exists a smooth volume preserving vector field positively transverse to $\xi_\pm$.
\end{enumerate}
In particular, $\xi_\pm$ are weakly semi-fillable, hence tight.
\end{prop}

The Eliashberg--Thurston theorem naturally produces positive contact pairs approximating suitable foliations, and the pairs approximating taut foliations are automatically taut.

\section{From contact pairs to foliations}

In this section, we describe the main result of~\cite{Mas24} on the construction of foliations from contact pairs. As before, $M$ denotes a closed, oriented, connected $3$-manifold.

    \subsection{Main result}

\subsubsection{Statement}

Let $(\xi_-, \xi_+)$ be a positive contact pair on $M$, and let $Z$ be a smooth vector field positively transverse to both $\xi_\pm$.

\begin{thm} \label{thm:mainthm}
    Assume that at least one of $\xi_\pm$ is tight. Then there exists a $C^0$-foliation $\mathcal{F}$ transverse to $Z$. If moreover $(\xi_-, \xi_+)$ is taut, then $\mathcal{F}$ is taut.
\end{thm}

Recall that the contact structures in a taut pair are automatically tight. Therefore, we obtain:

\begin{cor}
    Assume that $M \neq S^1 \times S^2$. Then $M$ admits a taut (topological) foliation if and only if it admits a taut contact pair.
\end{cor}

Since contact structures are inherently more flexible than foliations, this reformulation of the existence of taut foliations in terms of contact geometry will allow us to perform constructions that are a priori not available for foliations, see below.

\subsubsection{Brief overview of the proof}

The proof of Theorem~\ref{thm:mainthm} roughly goes as follows:
\begin{itemize}
    \item First, we construct a locally integrable plane field $\eta$ from the pair $(\xi_-, \xi_+)$ using the flow of a vector field $X$ spanning the intersection $\xi_- \cap \xi_+$. This vector field vanishes exactly at the points where $\xi_-$ and $\xi_+$ coincide. For generic contact pairs, the locus $\Delta$ of such points is nice enough, and the singularities of $X$ along $\Delta$ are not too wild. We note that the plane field $\eta$ is canonically associated with the pair, and depends continuously on it.
    
    \item The second and hardest step is to integrate the (continuous!) plane field $\eta$. Since it fails to be uniquely integrable in general, the construction requires a rather ad hoc strategy. Following the approach of Burago--Ivanov~\cite{BI08}, we prove that for a generic pair $(\xi_-, \xi_+)$, the associated plane field is tangent to a \emph{branching foliation}. Namely, the leaves are allowed to intersect, but not to cross topologically. The presence of singularities and closed orbits for $X$ greatly complicates the task, and the tightness condition on $\xi_\pm$ prevents certain obstructions from appearing.
    
    \item The last step is to obtain a genuine $C^0$-foliation from the branching foliation tangent to $\eta$. The strategy already appears in~\cite{BI08} and consists of separating the leaves of the branching foliation, at the expense of a small $C^0$ perturbation of $\eta$. This perturbation heavily depends on the structure of the leaves of the branching foliation, which itself depends on multiple intractable choices made throughout the construction. Therefore, while we maintain some control on the tangent plane field of the resulting foliation, it is in general impossible to keep track of any information pertaining to its leaves.
\end{itemize}

We provide more details on these steps in the following sections.

    \subsection{Unstable plane field}

We fix a positive contact pair $(\xi_-, \xi_+)$ on $M$, and we write
$$\Delta \coloneqq \{ p \in M : \xi_-(p) = \xi_+(p)\}.$$
We then consider a vector field $X$ such that 
\begin{itemize}
    \item $X$ vanishes exactly along $\Delta$,
    \item Away from $\Delta$, $\xi_- \cap \xi_+ = \langle X \rangle$.
\end{itemize}
Note that such a vector field is not unique, but its direction and normalization can be fixed as follows. First, we consider an arbitrary volume form $\mathrm{dvol}$ on $M$, and we denote by $\alpha_\pm$ the unique contact forms for $\xi_\pm$ satisfying \begin{align} \label{eq:balanced}
\alpha_+ \wedge d\alpha_+ = - \alpha_- \wedge d\alpha_- = \mathrm{dvol}.
\end{align}
Then, we define $X$ as the unique vector field satisfying $\iota_X \mathrm{dvol} = \alpha_- \wedge \alpha_+$. One easily checks that $X$ is independent of the choice of $\mathrm{dvol}$.

\smallskip

Since $X$ is tangent to both $\xi_-$ and $\xi_+$, these contact structures twist along $X$ in opposite directions. The idea to obtain an integrable plane field from the contact pair is to flow them for a very long time along $X$, and prove that the resulting plane fields converge to a given continuous plane field. The latter would be tangent to $X$, and invariant under its flow, so it would be locally integrable at least away from $\Delta$ by Lemma~\ref{lem:tangentint}. This strategy was implemented by Mitsumatsu~\cite{M95} and Eliashberg--Thurston~\cite{ET} for transverse contact structures, and in~\cite{CF11} for suitable positive pairs. In~\cite{Mas24} we generalized this procedure to an arbitrary positive contact pair, and showed:

\begin{prop}[See~\cite{Mas24}, Theorem A] \label{prop:limits}
    Let $(\xi_-, \xi_+)$ be any positive contact pair on $M$, with associated vector field $X$, and define 
    $$\xi^t_\pm \coloneqq (\varphi^t_X)_* \xi_\pm.$$
    Then both limits $\lim_{t \rightarrow +\infty} \xi^t_\pm$ exist and coincide with a continuous plane field $\eta$. Moreover, the assignment $(\xi_-, \xi_+) \mapsto \eta$ is continuous (for the $C^\infty$ topology on the source and the $C^0$ topology on the target). 
\end{prop}

We call $\eta$ the \textbf{unstable plane field} of the pair $(\xi_-, \xi_+)$, in analogy with the case where $X$ is a (projectively) Anosov flow. Such flows always lie at the intersection of a \emph{transverse} contact pair, and the flowing procedure of the previous proposition converges to the weak-unstable plane field of the flow.

\smallskip

To prove Proposition~\ref{prop:limits}, we use a rather indirect route. We actually show that there exists a unique plane field $\eta$ satisfying
\begin{itemize}
    \item $\eta$ is tangent to $X$ and invariant under the flow of $X$,
    \item $\eta$ coincides with $\xi_\pm$ along $\Delta$,
    \item $\eta$ is ``sandwiched'' between $\xi_-$ and $\xi_+$ in the appropriate quadrant away from $\Delta$.
\end{itemize}
To find such a plane field $\eta$, we choose contact forms $\alpha_\pm$ for $\xi_\pm$ satisfying~\eqref{eq:balanced}, and we look for a map $\sigma : M \rightarrow \R$ such that the $1$-form
$$\alpha \coloneqq  e^\sigma \alpha_+ - e^{-\sigma} \alpha_-$$
defines a plane field $\eta = \ker \alpha$ satisfying the above conditions. This turns out to be equivalent to an equation of the form
\begin{align} \label{eq:sigma}
    X \cdot \sigma + \sinh(2 \sigma) = g,
\end{align}
for some smooth function $g : M \rightarrow \R$ determined by the contact pair. Restricted to a flow line $(\varphi^t_X(p))_t$ of $X$, this reduces to an ODE with a very unstable behavior in negative time. In particular, there is a unique initial value for this ODE such that the corresponding solution is defined for all times and stays bounded. We then define $\sigma(p)$ to be this initial value. Standard results about the dependence of solutions of ODEs on initial values and parameters imply that this procedure defines a continuous map $\sigma : M \rightarrow \R$ which is differentiable along $X$ and solves~\eqref{eq:sigma}. Moreover, this map depends continuously on $(\xi_-, \xi_+)$ for similar reasons. Notice that the specific shape of $\Delta$ is irrelevant in this proof.

    \subsection{Integrating the unstable plane field}

By the previous section, we can associate to each positive contact pair $(\xi_-, \xi_+)$ a pair $(X, \eta)$ made of 
\begin{itemize}
    \item A smooth vector field $X$ tangent to both $\xi_\pm$ and vanishing along $\Delta$,
    \item A continuous plane field $\eta$ sandwiched between $\xi_-$ and $\xi_+$, which is tangent to $X$ and invariant under its flow.
\end{itemize}
We call the pair $(X, \eta)$ a \textbf{polarized vector field}. Away from $\Delta$, $\eta$ is locally integrable by Lemma~\ref{lem:tangentint}. However, the local behavior of $\eta$ near points in $\Delta$ is much less clear, as $\Delta$ can be quite complicated in general. Nevertheless, for a \emph{generic} pair $(\xi_-, \xi_+)$, $\Delta$ is a smooth embedded link in $M$, and the vector field $X$ has tractable singularities along $\Delta$. It has the following structure: $\Delta$ is transverse to $\xi_\pm$ away from finitely many points $Q \subset \Delta$ where it is tangent to them with quadratic tangencies. The connected components of $\Delta \setminus Q$ are of two types:
\begin{itemize}
    \item Arcs or circles made of normally elliptic singularities, i.e., $1$-parameter families of $2$-dimensional source singularities for $X$,
    \item Arcs or circles made of normally hyperbolic singularities, i.e., $1$-parameter families of $2$-dimensional saddle singularities for $X$.
\end{itemize}
The quadratic points in $Q$ lie at the junction between source and saddle arcs, and $\Delta$ looks like the movie of the elimination of an elliptic singularity with a hyperbolic singularity near those points. In particular, the singularities in $Q$ can be described as hybrids between source and saddle singularities.

The careful analysis of the behavior of $X$ and $\eta$ near $\Delta$ yields:

\begin{prop}[See~\cite{Mas24}, Theorem A]
    Under the previous assumptions, the continuous plane field $\eta$ is locally integrable \emph{everywhere} on $M$.
\end{prop}

However, $\eta$ might fail to be tangent to a foliation. In~\cite[Example 2.2.9]{ET}, the authors exhibit a pair of transverse contact structures whose corresponding unstable plane field is not tangent to a foliation. Nevertheless, the vector field $X$ itself provides some extra structure that can be used to patch together local surfaces tangent to $\eta$; since $X$ has some potentially very nontrivial global dynamics, much care is needed to ensure that this procedure yields \emph{noncrossing} surfaces. This is exactly what the strategy described in~\cite{BI08} manages to achieve. 

    \subsection{Branching foliations: an overview}

Let $(X, \eta)$ be a polarized vector field obtained from a generic positive contact pair $(\xi_-, \xi_+)$. The plane field $\eta$ naturally comes with a coorientation induced by the contact pair. If $X$ has no singularities (i.e., $\xi_-$ and $\xi_+$ are transverse everywhere), then the strategy developed by Burago--Ivanov in~\cite{BI08} implies that $\eta$ is tangent to a \textbf{branching foliation.} 

\subsubsection{Definition}

A branching foliation tangent to $\eta$ is defined as a collection of immersed surfaces $f_i : \Sigma_i \looparrowright M$, $i \in I$, which are tangent to $\eta$ and maximal for the inclusion, which cover $M$, and which are allowed to intersect but not to \emph{topologically cross} each other (or themselves). Two such immersions $f_i$ and $f_j$, $i,j \in I$ have a \textbf{topological crossing} if there exists a $C^1$ curve $\gamma : [0,1] \rightarrow M$ contained in both surfaces, such that near $\gamma(0)$, $f_i(\Sigma_i)$ lies below $f_j(\Sigma_j)$ (with respect to the coorientation on $\eta$) and is not entirely contained in it, while near $\gamma(1)$, $f_j(\Sigma_j)$ lies below $f_i(\Sigma_i)$ and is not entirely contained in it. Here, we allow $i=j$ as the immersions are not assumed to be injective.

\subsubsection{An obstruction}

In our more general setup, $X$ is allowed to have singularities and closed orbits. In this context, the strategy of Burago--Ivanov needs to be adapted, and certain obstructions might appear. Those take the form of \textbf{disks of tangency}, defined as immersions $f : \overline{D} \looparrowright M$ of the closed $2$-disk which are tangent to $\eta$ and with boundary tangent to $X$. The presence of such objects might prevent the extension procedure of Burago--Ivanov from going through; roughly speaking, these disks of tangency are ``traps'' which cannot be glued to tangent surfaces with nonclosed boundary components. This boils down to the fact that the disk is the only compact surface whose universal cover has a closed boundary component. Fortunately, we have:

\begin{prop}[See~\cite{Mas24}, Proposition 2.8]
    If one of $\xi_\pm$ is tight, or if $\Delta$ has saddle singularities only, then $(X, \eta)$ has no disks of tangency.
\end{prop}

\subsubsection{Main technical result}

At the heart of the proof of Theorem~\ref{thm:mainthm} is the adaptation of the strategy of Burago--Ivanov to construct a branching foliation tangent to $\eta$. More precisely, we showed:

\begin{thm}
    Let $(\xi_-, \xi_+)$ be a generic positive contact pair on $M$, such that at least one of $\xi_\pm$ is tight, or so that the associated vector field $X$ has saddle singularities only. Then its unstable plane field $\eta$ is tangent to a branching foliation.
\end{thm}

While $\eta$ is canonical, the branching foliation that we construct \emph{is not}: it is obtained through an inductive procedure that requires various auxiliary choices at each step. While our method is constructive, it is essentially impossible to keep track of how these choices influence the final branching foliation. Moreover, $\eta$ might admit infinitely many distinct branching foliations tangent to it!

\subsubsection{Weak prefoliations}

To adapt the strategy of Burago--Ivanov and construct a branching foliation tangent to $\eta$, we inductively construct collections of immersed surfaces tangent to $\eta$, and we inductively extend them until they become maximal. During the construction, we also keep track of how the surfaces intersect, by recording the data of a (pre)order at points where they coincide. This ensures that the successive extensions won't create topological crossings. 

More precisely, we consider a class of immersions of surfaces in $M$ that we call \textbf{tiles}. Those are parametrized by surfaces with boundary and corners homeomorphic to contractible subsets of $\R^2$, which are saturated by the flow lines of $X$, and with controlled behavior near the singular set $\Delta$. 

Generalizing the \emph{prefoliations} from~\cite{BI08}, we then define a \textbf{weak prefoliation} as a collection of tiles $\mathscr{A}$, together with a partial preorder $\precsim$ on the set of \emph{marked} tiles in $\mathscr{A}$. Here, a marked tile is a tile together with a marked point in its domain of definition. The preorder $\precsim$ satisfies a certain number of natural axioms:
\begin{enumerate}
    \item[0.] Two marked tiles $f_\star$ and $g_\star$ are comparable by $\precsim$ if and only if they send their marked points to a common point in $M$. Moreover, if both $f_\star \precsim g_\star$ and $g_\star \precsim f_\star$ hold, then there exists a diffeomorphism $\varphi$ of their domains satisfying $f = g \circ \varphi$, and $\varphi$ sends the marked point of $f_\star$ to the marked point of $g_\star$. In that case, we say that $f_\star$ and $g_\star$ are \emph{geometrically equivalent}. However, we don't need to require that the converse holds: two geometrically equivalent marked tiles may not be equivalent with respect to $\precsim$.
    \item [1.] The preorder $\precsim$ refines the geometric order, in the sense that if $f_\star \precsim g_\star$, then $g_\star$ is geometrically above $f_\star$ near the point in $M$ where they meet, with respect to the coorientation on $\eta$.
    \item [2.] Finally, the preorder $\precsim$ is coherent along intersections. Roughly speaking, this means that the relative preorder between two marked tiles doesn't change when we slide the marked points along a curve contained in their intersection.
\end{enumerate}

These properties ensure that two tiles in $\mathscr{A}$ don't topologically cross. Moreover, this notion is flexible enough to inductively extend the tiles. In fact, in~\cite{Mas24}, we construct a sequence of weak prefoliations 
$$\varnothing = \mathscr{A}_0 \hookrightarrow \mathscr{A}_1 \hookrightarrow \dots \hookrightarrow \mathscr{A}_n \hookrightarrow \mathscr{A}_{n+1} \hookrightarrow \dots,$$
such that each extension increases the sizes of the tiles significantly. Passing to the (co)limit, we obtain a collection of \emph{maximal} tiles without boundaries which cover $M$ and without topological crossings. In other words, the resulting object is a branching foliation tangent to $\eta$. We note that the extension procedure is particularly tricky, since one has to extend not only the tiles themselves but also the preorder structure; the singularities of $X$ and its global dynamics make the extension of the preorder a particularly delicate task.

\subsubsection{Separating the leaves}

Now that we have constructed a branching foliation tangent to $\eta$, we obtain a genuine $C^0$-foliation by separating the leaves of the branching foliation. More precisely, we use:

\begin{thm}[See~\cite{BI08}, Section 7.1]
    If $\eta$ is a $C^0$ plane field tangent to a branching foliation, then $\eta$ is $C^0$-approximated by plane fields tangent to foliations. Moreover, for each approximating foliation $\mathcal{F}$, there exists a continuous surjective map $h: M \rightarrow M$ such that the restriction of $h$ to every leaf of $\mathcal{F}$ is a $C^1$ immersion tangent to $\eta$.
\end{thm}

Here, the map $h$ is collapsing the leaves of $\mathcal{F}$ to images of leaves of the branching foliation. We note that the branching foliation previously constructed only has contractible leaves, some of which can be covers of closed surfaces in $M$. In that case, the leaf is collapsed, and not ``separated from itself''. 

\section{Applications and conjectures}

We now present some applications of our main result, as well as some speculations about the $L$-space conjecture.

    \subsection{Surgeries transverse to taut foliations}

Our first application is a construction of surgeries on taut foliations along transverse knots (or links).

\subsubsection{Surgery along transverse knots}

Let $M$ be a closed, connected, oriented $3$-manifold as before, and let $K \subset M$ be a framed knot in $M$. If $s \in \Q$ is a rational number, we denote by $M_K(s)$ the manifold obtained from $M$ by performing a Dehn surgery of slope $s$ along $K$.

If $\mathcal{F}$ is a foliation on $M$ transverse to $K$, performing a nontrivial Dehn surgery along $K$ typically destroys $\mathcal{F}$, in the sense that $\mathcal{F}$ does not induce a foliation on $M_K(s)$. It rather induces a singular foliation, which is singular along the image of $K$ in $M_K(s)$. Instead, it is possible to perform a surgery on an approximating contact pair and obtain a taut contact pair on $M_K(s)$, provided that $s$ is large enough. In~\cite{Mas24}, we proved:

\begin{thm}[See~\cite{Mas24}, Theorem H] \label{thm:transfol}
Let $\mathcal{F}$ be a taut aspherical $C^0$-foliation on $M$, and $K \subset M$ be a framed knot transverse to $\mathcal{F}$. There exists $s_0 = s_0(\mathcal{F},K) \geq 0$ such that for every rational number $s \in \Q$ satisfying $\vert s \vert \geq s_0$, $M_K(s)$ carries a taut $C^0$-foliation $\mathcal{F}'$. Moreover, the image of $K$ in $M_K(s)$ is transverse to $\mathcal{F}'$.
\end{thm}

As a corollary, we obtain a generalization of the main result of~\cite{LR14}, with a somewhat easier and more natural proof. 

Let $K \subset S^3$ be a nontrivial knot. We denote by $\mathcal{S}_K \subset \Q$ the set of rational slopes $s \in \Q$ such that there exists a taut $C^0$-foliation on $S^3 \setminus K$ which intersects the boundary of a tubular neighborhood of $K$ transversally along a foliation by closed curves of slope $s$. By a celebrated theorem of Gabai~\cite{G87}, $\mathcal{S}_K$ contains $0$. By~\cite[Theorem 1.1]{LR14}, $\mathcal{S}_K$ contains a neighborhood of $0$. We actually get:

\begin{cor} $\mathcal{S}_K$ is an open subset of $\Q$.
\end{cor}

We obtain this corollary as a rather brutal application of Theorem~\ref{thm:transfol} for which we have little control on the number $s_0$ in such generality. Instead, it would be interesting to construct a concrete contact pair adapted to $K$ and carefully analyze the slopes of the characteristic foliations of the contact structures on the boundary of a (potentially large) solid torus containing $K$. 

\subsubsection{Surgery along transverse flows}

The ideas from the previous paragraph can be generalized to other situations, in particular in the presence of a suitable flow transverse to the taut foliation.

A folklore conjecture, attributed to Gabai and Mosher in the \emph{finite depth} case, states that every cooriented taut foliation on a closed $3$-manifold is (almost) transverse to a pseudo-Anosov flow. Roughly speaking, a pseudo-Anosov flow is a hyperbolic flow with some \emph{singular} orbits where the weak foliations exhibit \emph{prong singularities}. See also~\cite{C06} for some related results and notions, in particular the \emph{pseudo-Anosov packages}.

Motivated by this philosophy, it is natural to fix a transitive pseudo-Anosov flow $\Phi$ (e.g., the suspension of the monodromy of a hyperbolic fibered link in $S^3$) and study its transverse taut foliations. To do so, we first remove tubular neighborhoods of the singular orbits, as well as finitely many more orbits to achieve the ``no perfect fits'' technical condition. The latter implies that on the resulting manifold with boundary $\check M$, every taut foliation transverse to the flow is supported by a single branched surface coming from a \emph{veering triangulation}~\cite{Z24}.

The boundary components of $\check M$ are tori with preferred framings induced by the flow, and we can encode the multislopes along those by an element of $(\R \cup \{\pm \infty\})^n$, where $n$ denotes the number of boundary components of $\check M$. The $\pm \infty$ slopes are called the \emph{degeneracy slopes}. We now define the set $\mathcal{Z} \subset \R^n$ as the set of (nondegeneracy) multislopes that can be realized by a taut foliation transverse to the flow in $\check M$. Borrowing terminology from~\cite{CW}, we call $\mathcal{Z}$ a \textbf{ziggurat}. See Figure~\ref{fig}.

In~\cite{MZ}, we will prove some structural properties for the set $\mathcal{Z}$. For simplicity, we state a weaker result under some technical (yet fairly general) assumption.

\begin{thm}[M.--Zung,~\cite{MZ}]
    Assume that there is no compact genus $0$ surface in $\check M$ with boundary in $\partial \check M$ which is transverse to $\Phi$ and with some boundary with degeneracy slope. Then the following hold:
    \begin{enumerate}
        \item (Closedness) $\mathcal{Z} \subset \R^n$ is closed.
        \item (Box-convexity) $\mathcal{Z}$ is \emph{box-convex}: if $s=(s_1, \dots, s_n) \in \mathcal{Z}$ and $t=(t_1, \dots, t_n) \in \mathcal{Z}$ satisfy $s_i \leq t_i$, $1 \leq i \leq n$, then 
        $$[s_1, t_1] \times \dots \times [s_n, t_n] \subset \mathcal{Z}.$$
        \item (Rationality) If $s=(s_1, \dots, s_n) \in \mathcal{Z}$ and $s_1 \in \R \setminus \Q$, then there is an $\epsilon > 0$ such that
        $$(s_1 - \epsilon, s_1 + \epsilon) \times \{(s_2, \dots, s_n)\} \subset \mathcal{Z}.$$
        \item (Rational fillings) Let $s \in \Q^n$. If $s \in \mathrm{int} \mathcal{Z}$, then the $s$-multislope Dehn filling of $\check M$ admits a taut foliation transverse to the flow $\Phi_s$ induced by $\Phi$. 
        
        Conversely, if the $s$-multislope Dehn filling of $\check M$ admits a taut foliation transverse to $\Phi_s$, and this foliation is different from the product foliation by spheres on $S^1 \times S^2$,\footnote{Otherwise, $\mathcal{Z}$ is reduced to a point!} then $s \in \mathrm{int} \mathcal{Z}$.
        \item (Slippery points) If $s = (s_1, \dots, s_n) \in \mathcal{Z}$ is an accumulation point of extremal points, then either $s \in \Q^n$ and the $s$-multislope Dehn filling of $\check M$ has an $S^1 \times S^2$ summand, or at least two entries in $s$ are irrational and the Dehn filling of $\check M$ along the boundaries with rational entries has a $\mathbb{T}^2 \times [0,1]$-summand.
    \end{enumerate}
\end{thm}

In the last item, the case of the $S^1 \times S^2$-summand should not be a surprise, since it corresponds to the only manifold carrying a taut foliation that cannot be approximated by contact structures. The case of the $\mathbb{T}^2 \times [0,1]$-summand might seem a bit more mysterious, and arises when considering irrational foliations on $\mathbb{T}^2 \times [0,1]$ (namely, the product of an irrational foliation on $\mathbb{T}^2$ with $[0,1]$). Such a foliation cannot be approximated by a positive (or negative) contact structure whose boundary slopes are both ``better'' than the foliation on the two boundary components.

The main technical ingredients of the proof are:

\begin{itemize}
    \item A generalization of the Eliashberg--Thurston theorem for manifolds with boundary, with a \emph{precise control} on the slopes and characteristic foliations of the contact approximations along the boundary,
    \item A generalization of Theorem~\ref{thm:mainthm} for manifolds with boundary, for contact pairs with a specific form along the boundary.
\end{itemize}

\begin{figure}[!ht]
    \centering
    \includegraphics[width=1\linewidth]{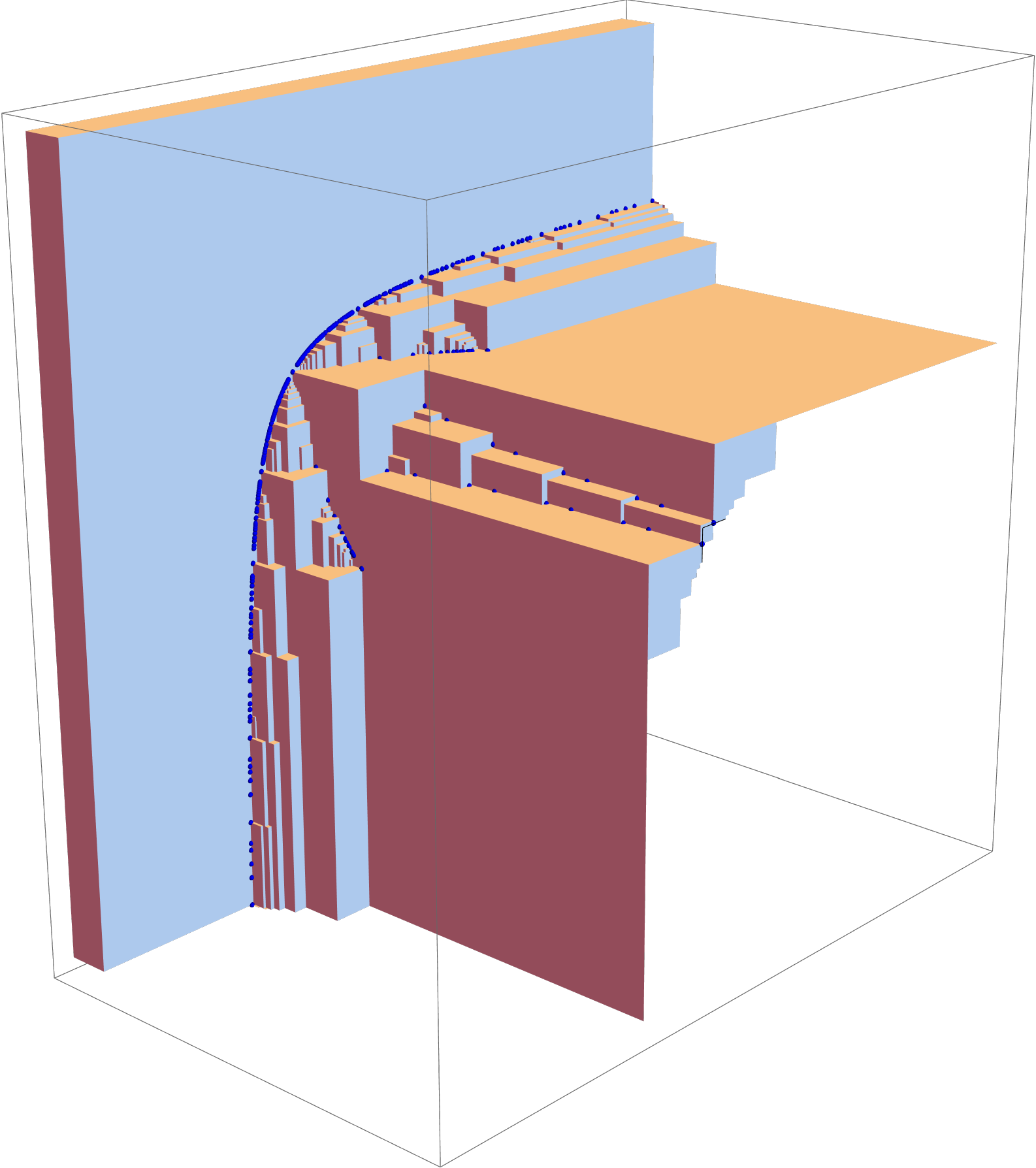}
    \caption{A ziggurat in $\R^3$ for the complement of the fibered link $L6a5$. The blue dots correspond to $S^1 \times S^2$ surgeries. The hyperbola  corresponds to ``matching slopes'' on a $\mathbb{T}^2 \times [0,1]$ surgery. Image courtesy of Jonathan Zung.}
    \label{fig}
\end{figure}

    \subsection{Towards a contact $L$-space conjecture}

In this section, we propose some conjectures and speculations around the celebrated $L$-space conjecture, and we outline a contact viewpoint towards this conjecture.

\subsubsection{The $L$-space conjecture}

It is an open problem to determine which irreducible rational homology $3$-spheres carry a taut foliation. The famous $L$-space conjecture provides a hypothetical answer:

\begin{conj*}[$L$-space conjecture~\cite{BGW,J15}]
Let $M$ be an irreducible rational homology $3$-sphere. The following are equivalent:
\begin{enumerate}
    \item $M$ carries a cooriented taut foliation,
    \item $\pi_1(M)$ is left-orderable,
    \item $M$ is not an $L$-space.
\end{enumerate}
\end{conj*}

The second item can be thought of as an algebraic characterization, while the third one comes from gauge theory. Heegaard Floer homology is an invariant of $3$-manifolds constructed by Ozsváth and Szabó in the early 2000s, and a $3$-manifold $M$ is an $L$-space if its `hat' Heegaard Floer homology is as small as possible, namely
$$\mathrm{rank} \widehat{HF}(M) = \vert H_1(M, \Z)\vert,$$
or alternatively if its reduced Heegaard Floer homology vanishes, $HF^\mathrm{red}(M) = 0$. Lens spaces are examples of $L$-spaces.

In~\cite{OS04}, Ozsváth and Szabó proved the implication $1 \implies 3$ by showing that the \emph{contact class} of a contact approximation to the foliation provided by the Eliashberg--Thurston theorem is nonzero in $HF^\mathrm{red}$. Interestingly, contact structures seem to have better interactions with gauge theory than foliations.

\smallskip

One key difficulty to reverse Ozsváth--Szabó's result and show that non-$L$-space irreducible rational homology spheres admit taut foliations is that very few general strategies are available to construct taut foliations directly from algebraic or gauge-theoretic assumptions. To our knowledge, all the known cases where the equivalence $1 \iff 3$ holds were obtained by rather ad hoc arguments, by considering a class of $3$-manifolds and independently determining which are $L$-spaces and which admit a taut foliation. See for instance~\cite{BC17, R17, RR17, HRRW20} for all graph manifolds, \cite{D20} for a large census of hyperbolic manifolds, and~\cite{K20, K23, S23a, S23b, S23c} for $3$-manifolds obtained by Dehn surgeries on suitable knots and links in $S^3$. Moreover, these approaches are quite combinatorial and perhaps a bit miraculous in nature. Instead, we propose a general program to construct Reebless and taut foliations that remains closer to the argument of Ozsváth--Szabó.

\subsubsection{Constructing Reebless foliations}

While we are able to construct taut foliations from taut contact pairs, it would be desirable to construct `rigid' foliations from more general pairs, under more flexible hypotheses. For instance, it would be natural to consider positive contact pairs made of tight contact structures, called \textbf{tight} positive contact pairs. Recall that Reebless foliations are approximated by such tight pairs. We expect the following to hold:

\begin{conj}
If $M \neq S^1 \times S^2$, then $M$ carries a Reebless $C^0$-foliation if and only if it carries a tight positive contact pair.
\end{conj}

The main difficulty comes from the lack of regularity of the unstable plane field $\eta$, and the fact that Reeblessness is \emph{not} a natural property for plane fields (see the \emph{phantom Reeb components} from~\cite{CKR19}). Notice that this conjecture would immediately imply that on an \emph{atoroidal} closed $3$-manifold, the existence of a tight positive contact pair is equivalent to the existence of a taut foliation.

\subsubsection{Removing the positivity condition}

Our construction of foliations from contact structures requires our two contact structures to be in a specific geometric position with respect to each other, namely, to form a positive pair. However, it is hard to construct such positive pairs in practice. It would be more convenient to work with pairs of contact structures which are not necessarily positive, and impose conditions that are \emph{intrinsic} for both contact structures independently, without any condition on their geometric interaction.

A contact pair whose contact structures are homotopic as plane fields can always be deformed into a positive contact pair when one of its contact structures is overtwisted. However, there is little hope of constructing geometrically interesting foliations from such a pair. Upgrading tight contact pairs to positive ones seems much more difficult. At the very least, we should further assume that the contact structures are homotopic as oriented plane fields, but this might not be enough. Indeed, Lin recently remarked that the existence of a taut foliation on a rational homology sphere imposes more constraints on the (monopole) Floer homology of the manifold, see~\cite{L23}. In particular, the $\mathbb{F}[U]$-module $HM_*$ has a \emph{direct $\mathbb{F}$-summand}, where $\mathbb{F} = \Z \slash 2\Z$. Moreover, the \emph{contact invariants} of the contact approximations of the foliation have to pair to $1$, for the natural perfect pairing
$$\langle \, \cdot \, , \, \cdot \, \rangle : HM_*(M) \otimes HM_*(-M) \longrightarrow \mathbb{F}$$
induced by the Poincar\'{e} duality isomorphism $HM^*(M) \cong HM_*(-M)$. Motivated by this result, we say that a contact pair $(\xi_-, \xi_+)$ on $M$ is \textbf{algebraically tight} if the contact invariants $c(\xi_\pm) \in HM(\mp M)$ satisfy
$$\big \langle c(\xi_-), c(\xi_+)\big\rangle = 1.$$

Notice that if $(\xi_-, \xi_+)$ is algebraically tight, then $\xi_-$ and $\xi_+$ are tight and homotopic as oriented plane fields. Indeed, the contact class of an overtwisted contact structure vanishes, and the grading of the contact class corresponds to the homotopy class of the contact structure (as an oriented plane field). We propose:

\begin{conj}
Let $(\xi_-, \xi_+)$ be an algebraically tight contact pair on $M$. Then $(\xi_-, \xi_+)$ is homotopic through contact pairs to a positive one.
\end{conj}

Here, we think of the contact classes as potential \emph{obstructions} to deforming a tight contact pair into a positive one. One possible approach to this conjecture would be to consider the locus $\Delta_- \subset M$ where $\xi_-$ and $\xi_+$ coincide with \emph{opposite} orientations. It is generically an embedded link which is null-homologous if $\xi_-$ and $\xi_+$ are homotopic. Perhaps one could study the link Floer or sutured Floer homology of the pair $(M, \Delta_-)$ and use the algebraic tightness hypothesis to precisely understand the topology of $\Delta_-$. Additionally, there could be `elementary moves' that would inductively simplify $\Delta_-$. Some obstruction to performing these moves could be determined by Floer-theoretic invariants of the contact pair.

\smallskip

It remains to understand how to construct suitable tight contact structures on rational homology spheres. This is of course a very delicate problem, but perhaps more tractable than constructing taut foliations. We ask:

\begin{question}[Contact realization] Assume that $M$ is an irreducible rational homology sphere.
\begin{itemize}
    \item Which nonzero elements in the Floer homology of $M$ can be realized as the contact class of a (necessarily tight) contact structure?
    \item If the Floer homology of $M$ has a direct $\mathbb{F}$-summand, does $M$ carry an algebraically tight contact pair?
\end{itemize}
\end{question}

Assuming the previous two conjectures, a positive answer to the second item of this question would imply:

\begin{conj}
If $M$ is an irreducible rational homology sphere such that $HM_*(M)$ has a direct $\mathbb{F}$-summand, then $M$ carries a Reebless $C^0$-foliation.
\end{conj}

Recently, Alfieri and Binns verified the existence of a direct $\mathbb{F}$-summand in Floer homology for a large class of non-$L$-space irreducible rational homology spheres~\cite{AB24}.

Note that it is not yet known if every non-$L$-space irreducible rational homology sphere admits \emph{any} tight contact structure!


\bibliographystyle{amsplain}
\bibliography{bibli}

@article{Z21,
    AUTHOR = {Zung, Jonathan},
     TITLE = {Reeb flows transverse to foliations},
   JOURNAL = {Geometry \& Topology},
    VOLUME = {28},
      YEAR = {2024},
    NUMBER = {8},
     PAGES = {3661--3695}
}

@misc{Mas24,
    author = {Massoni, Thomas},
    title = {Taut foliations and contact pairs in dimension three},
    note = {\url{https://arxiv.org/abs/2405.15635}},
    year = {2024}
}

@book{ET,
    title = {Confoliations},
    author = {Eliashberg, Yakov and Thurston, William},
    isbn = {978-0-8218-0776-7},
    series = {University Lecture Series},
    volume = {13},
    year = {1998},
    publisher = {American Mathematical Society}
}

@article{CF11,
    author = "Colin, Vincent and Firmo, Sebasti{\~a}o",
    title = "{Paires de structures de contact sur les vari{\'e}t{\'e}s de dimension trois.} ({French})",
    journal = "Algebraic {\&} Geometric Topology",
    volume = "11",
    number = "5",
    pages = "2627--2653",
    year = "2011"
}

@article{M95,
    author = "Mitsumatsu, Yoshihiko",
    title = "{Anosov flows and non-Stein symplectic manifolds}",
    journal = "Annales de l'Institut Fourier",
    volume = "45",
    number = "5",
    pages = "1407--1421",
    year = "1995"
}

@article{BI08,
    author = "Burago, Dmitri and Ivanov, Sergei",
    title = "{Partially hyperbolic diffeomorphisms of $3$-manifolds with Abelian fundamental groups}",
    journal = "Journal of Modern Dynamics",
    volume = "2",
    number = "4",
    pages = "541--580",
    year = "2008"
}

@article{LR14,
	author = "Li, Tao and Roberts, Rachel",
	journal = "Pacific Journal of Mathematics",
	pages = "149--168",
	title = "{Taut foliations in knot complements}",
    number = "1",
	volume = "269",
	year = "2014"
}

@article{G87,
	author = "Gabai, David",
	journal = "Journal of Differential Geometry",
	pages = "479--536",
	title = "{Foliations and the topology of $3$-manifolds. III}",
    number = "3",
	volume = "26",
	year = "1987"
}

@book{CC1,
    title = {Foliations {I}},
    author = {Candel, Alberto and Conlon, Lawrence},
    isbn = {978-0-8218-0809-2},
    series = {Graduate Studies in Mathematics},
    volume = {23},
    year = {2000},
    publisher = {American Mathematical Society}
}

@book{CC2,
    title = {Foliations {II}},
    author = {Candel, Alberto and Conlon, Lawrence},
    isbn = {978-0-8218-0881-8},
    series = {Graduate Studies in Mathematics},
    volume = {60},
    year = {2003},
    publisher = {American Mathematical Society}
}

@article{KR17,
	author = "Kazez, William and Roberts, Rachel",
	journal = "Geometry \& Topology",
	pages = "3601--3657",
	title = "{$C^0$ approximations of foliations}",
    number = "6",
	volume = "21",
	year = "2017"
}

@article{CKR19,
	author = "Colin, Vincent and Kazez, William and Roberts, Rachel",
	journal = "Communications in Analysis and Geometry",
	pages = "357--375",
	title = "{Taut foliations}",
    number = "2",
	volume = "27",
	year = "2019"
}

@article{KR19,
	author = "Kazez, William and Roberts, Rachel",
	journal = "Algebraic \& Geometric Topology",
	pages = "2763--2794",
	title = "{$C^{1,0}$ foliation theory}",
    number = "6",
	volume = "19",
	year = "2019"
}

@inproceedings{CH20,
    author = "Colin, Vincent and Honda, Ko",
    title = "{Foliations, contact structures and their interactions in dimension three}",
    booktitle = "Surveys in $3$-Manifold Topology and Geometry",
    series = "Surveys in Differential Geometry",
    volume = "25",
    pages = "71--101",
    publisher = "International Press",
    editor = "Agol, Ian and Gabai, David",
    year = "2020"
}

@article{B16b,
	author = "Bowden, Jonathan",
	journal = "Journal of Differential Geometry",
	pages = "219--237",
	title = "{Contact perturbations of Reebless foliations are universally tight}",
    number = "2",
	volume = "104",
	year = "2016"
}

@article{C02,
	author = "Colin, Vincent",
	journal = "Topology",
	pages = "1017--1029",
	title = "{Structures de contact tendues sur les vari\'{e}t\'{e}s toroidales et approximation de feuilletages sans composante de Reeb}",
    number = "5",
	volume = "41",
	year = "2002"
}

@article{E07,
	author = "Etnyre, John",
	journal = "Mathematical Research Letters",
	pages = "775--779",
	title = "{Contact structures on $3$-manifolds are deformations of foliations}",
    number = "5",
	volume = "14",
	year = "2007"
}

@article{E89,
	author = "Eliashberg, Yakov",
	journal = "Inventiones mathematicae",
	pages = "623--637",
	title = "{Classification of overtwisted contact structures on $3$-manifolds}",
    number = "3",
	volume = "98",
	year = "1989"
}

@article{V16,
    author = "Vogel, Thomas",
    title = "{On the uniqueness of the contact structure approximating a foliation}",
    journal = "Geometry \& Topology",
    volume = "20",
    number = "5",
    pages = "2439--2573",
    year = "2016"
}

@article{C01,
    author = "Calegari, Danny",
    title = "{Leafwise smoothing laminations}",
    journal = "Algebraic \& Geometric Topology",
    volume = "1",
    pages = "579--587",
    year = "2001"
}

@book{C07,
    title = {Foliations and the Geometry of $3$-Manifolds},
    author = {Calegari, Danny},
    isbn = {978-0-1985-7008-0},
    series = {Oxford Mathematical Monographs},
    year = {2007},
    publisher = {Oxford University Press}
}

@book{G09,
    title = {An Introduction to Contact Topology},
    author = {Geiges, Hansj\"{o}rg},
    isbn = {978-0-521-86585-2},
    series = {Cambridge Studies in Advanced Mathematics},
    volume = {109},
    year = {2009},
    publisher = {Cambridge University Press}
}

@article{BGW,
	author = "Boyer, Steven and Gordon, Cameron McA and Watson, Liam",
	journal = "Mathematische Annalen",
	pages = "1213--1245",
	title = "{On {L}-spaces and left-orderable fundamental groups}",
    number = "4",
	volume = "356",
	year = "2013"
}

@inproceedings{J15,
    author = "Juhász, András",
    title = "{A survey of Heegaard Floer homology}",
    booktitle = "New ideas in low dimensional topology",
    series = "Series on Knots and Everything",
    volume = "56",
    pages = "237--295",
    editor = "Kauffman, Louis and Manturov, Vassily",
    publisher = "World Scientific",
    year = "2015"
}

@article{OS04,
    author = "Ozsv\'{a}th, Peter and Szab\'{o}, Zolt\'{a}n",
    title = "{Holomorphic disks and genus bounds}",
    journal = "Geometry \& Topology",
    volume = "8",
    number = "1",
    pages = "311--334",
year = {2004}
}

@article{HRRW20,
    author = "Hanselman, Jonathan and Rasmussen, Jacob and Rasmussen, Sarah Dean and Watson, Liam",
    title = "{{L}-spaces, taut foliations, and graph manifolds}",
    journal = "Compositio Mathematica",
    volume = "156",
    number = "3",
    pages = "604--612",
year = {2020}
}

@article{R17,
    author = "Rasmussen, Sarah Dean",
    title = "{{L}-space intervals for graph manifolds and cables}",
    journal = "Compositio Mathematica",
    volume = "153",
    number = "5",
    pages = "1008--1049",
year = {2017}
}

@article{RR17,
    author = "Rasmussen, Jacob and Rasmussen, Sarah Dean",
    title = "{Floer simple manifolds and {L}-space intervals}",
    journal = "Advances in Mathematics",
    volume = "322",
    pages = "738--805",
year = {2017}
}

@article{BC17,
    author = "Boyer, Steven and Clay, Adam",
    title = "{Foliations, orders, representations, {L}-spaces and graph manifolds}",
    journal = "Advances in Mathematics",
    volume = "310",
    pages = "159--234",
year = {2017}
}

@article{K20,
    author = "Krishna, Siddhi",
    title = "{Taut foliations, positive 3-braids, and the {L}-space conjecture}",
    journal = "Journal of Topology",
    volume = "13",
    number = "3",
    pages = "1003--1033",
    year = {2020}
}

@article{K23,
    author = "Krishna, Siddhi",
    title = "{Taut foliations, braid positivity, and unknot detection}",
    journal = "Advances in Mathematics",
    volume = "470", 
    pages = "110233",
    year = {2025}
}

@article{S23a,
    author = "Santoro, Diego",
    title = "{{L}-spaces, taut foliations and the Whitehead link}",
    journal = "Algebraic \& Geometric Topology",
    volume = "24",
    number = "6",
    pages = "3455--3502",
    year = {2024}
}

@article{S23b,
    author = "Santoro, Diego",
    title = "{{L}-spaces, taut foliations and fibered hyperbolic two-bridge links}",
    journal = "Algebraic \& Geometric Topology",
    volume = "26",
    number = "3",
    pages = "1115--1154",
    year = {2026}
}

@article{S23c,
    author = "Santoro, Diego",
    title = "{Taut foliations from knot diagrams}",
    journal = "Advances in Mathematics",
    volume = "492",
    pages = "110906",
    year = {2026}
}

@article{L23,
    author = "Lin, Francesco",
    title = "{A remark on taut foliations and Floer homology}",
    journal = "Mathematical Research Letters",
    volume = "31",
    number = "6",
    pages = "1819--1825",
    year = "2025"
}

@article{D20,
    author = "Dunfield, Nathan",
    title = "{Floer homology, group orderability, and taut foliations of hyperbolic $3$-manifolds.}",
    journal = "Geometry \& Topology",
    volume = "24",
    number = "4",
    pages = "2075--2125",
year = {2020}
}

@misc{MZ,
    author = "Massoni, Thomas and Zung, Jonathan",
    title = "Ziggurats, taut foliations, and contact structures",
    note = "In preparation"
}

@article{CW,
    author = "Calegari, Danny and Walker, Alden",
    title = "{Ziggurats and rotation numbers}",
    journal = "Journal of Modern Dynamics",
    volume = "5",
    number = "4",
    pages = "711--746",
year = {2011}
}

@article{B16a,
    author = "Bowden, Jonathan",
    title = "{Approximating $C^0$-foliations by contact structures}",
    journal = "Geometric and Functional Analysis",
    volume = "26",
    number = "5",
    pages = "1255--1296",
    year = "2016"
}

@misc{BF03,
    author = {Bonatti, Christian and Franks, John},
    title = {A {H}\"{o}lder continuous vector field tangent to many foliations},
    note = {\url{https://arxiv.org/abs/math/0303291}},
    year = {2003}
}

@misc{AB24,
    author = {Alfieri, Antonio and Binns, Fraser},
    title = {Is the geography of {H}eegaard {F}loer homology restricted or the {L}-space conjecture false?},
    note = {\url{https://arxiv.org/abs/2404.00490}},
    year = {2024}
}

@misc{BM,
    author = {Bowden, Jonathan and Massoni, Thomas},
    title = {Topological invariance of {L}iouville structures for taut foliations and {A}nosov flows},
    note = {\url{https://arxiv.org/abs/2510.15325}},
    year = {2025}
}

@misc{Z24,
    author = {Zung, Jonathan},
    title = {Veering triangulations and transverse foliations},
    note = {\url{https://arxiv.org/abs/2411.00227}},
    year = {2024}
}

@article{C06,
    author = "Calegari, Danny",
    title = "{Promoting essential laminations}",
    journal = "Inventiones Mathematicae",
    volume = "166",
    pages = "583--643",
year = {2006}
}

\end{document}